\documentclass[reqno,12pt]{amsart}
\usepackage[centertags]{amsmath}
\usepackage{amsfonts,amssymb,mathrsfs}
\usepackage{colonequals}
\usepackage{mathtools}
\usepackage{graphicx}
\usepackage{placeins}
\usepackage{color}
\usepackage{enumerate}
\usepackage[latin1]{inputenc}
\usepackage{float}
\usepackage{longtable}
\usepackage{booktabs}
\usepackage{siunitx}
\usepackage{fullpage}

\newtheorem{theorem}{Theorem}
\newtheorem{corollary}{Corollary}
\newtheorem{proposition}{Proposition}
\newtheorem{lemma}{Lemma}
\newtheorem{conjecture}{Conjecture}%

\theoremstyle{definition}
\newtheorem{definition}{Definition}
\theoremstyle{remark}

\newcommand{\midd}{{\,|\,}}

\newcommand{\nmidd}{{\,\nmid\,}}

\begin{document}

\title{Matchable numbers}
\date{\today}

\author{Nathan McNew}
\address{Department of Mathematics, Towson University, Towson, MD 21252, USA}
\email{nmcnew@towson.edu}
\author{Carl Pomerance}
\address{Mathematics Department, Dartmouth College, Hanover, NH 03755, USA}
\email{carlp@math.dartmouth.edu}

\begin{abstract}
We say a natural number $n$ is matchable if there is a bijection from the set of $\tau(n)$ divisors of $n$ to the set $\{1,2,\dots,\tau(n)\}$, where corresponding numbers are relatively prime. We show that the set of matchable numbers has an asymptotic
density, which we compute, and we show that every squarefree number is matchable. We also present some related unsolved problems.
\end{abstract}

\maketitle

\section{Introduction}
\label{S:intro}

Given two finite sets of integers of the same cardinality, a bijection $\psi$ between them is said to be a coprime matching if for each $x$ in the domain of $\psi$,
$x$ and $\psi(x)$ are coprime.  Alternatively, one
can consider the bipartite graph from one set to the other where there is an edge whenever the numbers on the edge's vertices are coprime: a coprime matching is a perfect matching in this graph.

There have been several papers on coprime matchings over the years, mostly where the two sets are intervals of consecutive integers. For example, in \cite{PS}, Pomerance
and Selfridge proved a conjecture of D. J. Newman that there is always a coprime matching from $\{1,2,\dots,n\}$ to any other interval of $n$ consecutive integers. This was generalized by Bohman and Peng \cite{BP}
to some cases where the intervals are arbitrarily placed on the number line, and they
showed a connection to the notorious lonely runner conjecture. Their paper was subsequently improved in \cite{P1} and generalized to a counting problem in \cite{P2}, with further progress
by McNew \cite{M} and Sah and Sawhney \cite{SS}.

In this paper we consider a problem of Recam\'an \cite{Re}, where one of the sets continues to be an initial interval of consecutive integers, but the other set is generally not an interval, rather it is the set of divisors of a number $n$.
More precisely, for
a positive integer $n$ let $D(n)$ denote the set of divisors of $n$, and $\tau(n)=|D(n)|$. We say an integer $n$ is \textit{matchable} if there is a coprime matching between
$\{1,2,\dots,\tau(n)\}$ and $D(n)$.

One might think at first that every number is matchable, and this holds for $n=1,2,\dots,7$. However, 8 is not matchable, nor is any subsequent multiple of 4. The proof is easy:
If $4\midd n$, then at least $2/3$ of the members of $D(n)$ are even, but fewer than
$2/3$ of the members of $\{1,2,\dots,k\}$ are odd when $k>3$. Since even divisors must be mapped to odd numbers in a coprime matching, proper multiples of 4 are seen to be not matchable. This can be generalized to other primes as well, see below.

Say a number $n$ is an M-number if it is not divisible by any $p^p$ with $p$ prime. For example, every squarefree number is an M-number. It is easy to see that the set of M-numbers possesses an asymptotic density, which is
\[
\alpha=\prod_{p\,{\rm prime}}\left(1-\frac1{p^p}\right)=0.72199023441955\dots\, .
\]
Among the comments in \cite{Re} we find the conjecture of K\"onig and Alekseyev that every 
M-number is matchable, and few non-M-numbers are matchable,
and in particular, the asymptotic density of the set of matchable numbers is $\alpha$.
 In this paper we prove that the asymptotic density of the
 symmetric difference of the set of matchable numbers and the set of M-numbers is 0, so
 that the density of the set of matchable numbers is indeed $\alpha$.
 
 We agree with the conjecture of K\"onig and Alekseyev that every M-number is matchable.
 Towards a proof we show at least that every squarefree number is matchable.  
 Probably our techniques can be extended to the remaining M-numbers.
 
For each prime $p$ let
\[
M_p=\prod_{q\le p} q^{q-1},
\]
where $q$ runs over primes. It is easy to see that each $M_p$ is matchable. Indeed, for $1\le j\le\tau(M_p)$, we map $j$ to
\[
\psi(j):= \prod_{q\le p} q^{(j\bmod q)}.
\]
To see this note that since $\tau(M_p)=\prod_{q\le p}q$, the Chinese remainder theorem shows that each integer $j\in[1,\tau(M_p)]$ corresponds to a unique vector
$(j\bmod 2,j\bmod 3,\dots,j\bmod p)$. Further, for $q\le p$, $q\midd j$ if and only if $q\nmidd
\psi(j)$.

The set of M-numbers is precisely the set of all divisors of the numbers $M_p$
as $p$ varies. As noted in \cite{Re}, if one can prove that all of the divisors of a matchable number are themselves matchable, we would immediately have the corollary that
every M-number is matchable. Unfortunately, we did not find a way to make this elegant plan work.

\section{Preliminary results}

We generalize the result that if $4\midd n$ and $n>4$, then $n$ is not matchable.

\begin{proposition}
  Suppose $p^p|n$ for some prime $p$, and let $0\leq r<p$ be the remainder when $\tau(n)$ is divided by $p$. If $\tau(n)>r(p+1)$, then $n$ is not matchable.
\end{proposition}
\begin{proof}
Suppose $p^k$ is the largest power of $p$ dividing $n$, with $k\geq p$. We can partition the $
\tau(n)$ divisors of $n$ into $k+1$ sets each having size $\tau(n/p^k)$ according to the power to which $p$ appears as a factor. Only one of those sets will contain integers coprime to $p$, and so the total number of divisors of $n$ coprime to $p$ is $\frac{\tau(n)}{k+1} \leq \frac{\tau(n)}
{p+1}$. On the other hand, the number of integers in $[1,\tau(n)]$ divisible by $p$ is
\begin{equation}
\left\lfloor\frac{\tau(n)}p\right\rfloor = \frac{\tau(n) - r}{p}.\label{eq:taupratio} \end{equation} Since $\tau(n)>r(p+1)$ we find that $(\tau(n) -r)(p+1)>\tau(n)p$ and thus the quantity in
\eqref{eq:taupratio} is strictly greater than $\frac{\tau(n)}{p+1}$, the upper bound we just found for the number of divisors of $n$ coprime to $p$. Thus, there are too few divisors of $n$ coprime to
$p$ to match with these integers up to $\tau(n)$ divisible by $p$.
\end{proof}

\begin{corollary}
\label{cor1}
If $p^p|n$ for some prime $p$, and $\tau(n) \geq p^2$, then $n$ is not matchable.
\end{corollary}

Let $\omega(n)$ denote the number of different primes that divide $n$. Note that
$\tau(n)\ge 2^{\omega(n)}$, with equality if and only if $n$ is squarefree.

\begin{corollary}
\label{cor:ud}
The upper density of the set of matchable integers is at most $\alpha$.
\end{corollary}
\begin{proof}
It suffices to show that the set of matchable numbers that are not M-numbers has asymptotic density 0. The set of integers divisible by some $p^p$ for $p\ge N$ has density $\ll N^{-N}$. If $p^p\midd n$ and $n$ is matchable, then Corollary \ref{cor1} implies that $\omega(n) \le 2\log p/\log 2 $. So, if $n$ is matchable and not
an M-number, it is either divisible by some $p^p$ for $p\ge N$ or $\omega(n)\le 2\log N/\log 2$.

The counting function to $x$ of the latter numbers $n$ is $\ll x(\log\log x)^{2\log N/\log 2}/\log x$, so for $N$ fixed, the set has density 0. Putting the two together, the upper density is
$\le N^{-N}$, and since $N$ is arbitrary, the corollary is proved.
\end{proof}

We remark that if we let $N=\log\log x$ in the proof, we have the counting function to $x$ of the set of matchable numbers that are not M-numbers is $\le x/(\log x)^{1+o(1)}$ as
$x\to\infty$. This result is best possible since every number $27p$ with $p$ prime is matchable (as is easily checked), yet not an M-number.

Our plan is to first prove that squarefree numbers with at least 45 prime factors are
matchable, and then by a somewhat different method, we prove it for squarefree numbers with fewer than 45 prime factors. Finally, we extend the argument to M-numbers with
sufficiently many prime-power divisors not of the form $p^{p-1}$ and not divisible by the square of any large prime, and use a density argument to finish.

We conclude with some open problems and a discussion of strongly matchable numbers. These are numbers $n$ such that there is a coprime matching between $D(n)$ and every coprime arithmetic progression of $\tau(n)$ integers.

\section{Squarefree numbers with many prime factors}

\begin{lemma}\label{lem:half}
If $2n$ is matchable, then so is $n$.
\end{lemma}
\begin{proof}
We may assume $n$ is odd. A coprime matching for $2n$ pairs $D(2n)=D(n)\cup 2D(n)$ with
$[1,\tau(2n)]=[1,2\tau(n)]$. Since the even divisors $2D(n)$ must be paired with odd integers in
$[1,2\tau(n)]$, the odd divisors $D(n)$ are paired with the even integers $\{2,4,\ldots,2\tau(n)\}$. Dividing by 2, this gives a coprime matching of $D(n)$ with $[1,\tau(n)]$, so $n$ is matchable.
\end{proof}

\begin{theorem}
\label{th:sqfr}
Every squarefree number having at least $45$ prime factors is matchable.
\end{theorem}

We now introduce a notion that will be used to track error bounds in counting.

\begin{definition}\label{def:APcomb}
For an integer $k\geq 1$, we say that a set $S$ of integers is a \emph{$k$-AP combination} if it can be constructed as follows:
\begin{enumerate}
\item A single arithmetic progression is a $1$-AP combination.
\item If $S_1$ is a $k_1$-AP combination, $S_2$ is a $k_2$-AP combination, and $S_1\cap S_2=\emptyset$, then $S_1\cup S_2$ is a $(k_1+k_2)$-AP combination.
\item If $S_1$ is a $k_1$-AP combination, $S_2\subseteq S_1$ is a $k_2$-AP combination, then
$S_1\setminus S_2$ is a $(k_1+k_2)$-AP combination.
\end{enumerate}
\end{definition}

Note that since we don't assume the constituent arithmetic progressions are nonempty, any set
$S$ that is a $k$-AP combination is also a $k'$-AP combination for any $k'>k$.

\begin{lemma}\label{lem:APcomb-error}
If $S$ is a $k$-AP combination whose constituent arithmetic progressions all have common differences coprime to $d$, then the number of elements of $S$ divisible by $d$ is $|S|/d+\theta$ where $|\theta|\le k$.
\end{lemma}
\begin{proof}
We proceed by induction on the construction of $S$. For the base case, let $S$ be an arithmetic progression of length $m$ with common difference $q$ where $\gcd(d,q)=1$. If $m\geq d$, the elements of $S$ form a complete residue system modulo $d$, and among any $d$ consecutive terms, exactly one is divisible by $d$. So, the count of $d$-multiples is $m/d+\theta$ with $|\theta|
\le 1$.

For the inductive step, suppose $S=S_1\cup S_2$ with $S_1\cap S_2=\emptyset$. Then the count of $d$-multiples in $S$ equals
\[(|S_1|/d+\theta_1)+(|S_2|/d+\theta_2)=|S|/d+(\theta_1+\theta_2),\]
where $|\theta_1|\le k_1$, $|\theta_2|\le k_2$, so $|\theta_1+\theta_2|\le k_1+k_2$.

If $S_2\subseteq S_1$ and $S=S_1\setminus S_2$, then the count of $d$-multiples in $S$ equals
\[(|S_1|/d+\theta_1)-(|S_2|/d+\theta_2)=|S|/d+(\theta_1-\theta_2),\] where $|\theta_1-\theta_2|\le k_1+k_2$.
\end{proof}

\begin{lemma} \label{lem:partition}
Suppose $p_1<p_2<\cdots <p_j$ are primes with $p_1=2$ and $I$ is an interval of
$L$ consecutive integers where $2^j \mid L$ and $L\ge 4^j$. Then $I$ can be partitioned into
$2^j$ sets $A_{v,j}$ of size $L/2^{j}$,
parametrized by divisors $v$ of $m_j:=p_1p_2\dots p_j$, such that every member of $A_{v,j}$ is coprime to $v$.
 Moreover, each $A_{v,j}$ is a $2^{j-1}$-AP combination whose constituent AP's have common differences dividing $m_j$, and hence for any $d$ coprime to $m_j$, the number of elements of $A_{v,j}$ divisible by $d$ is within $2^{j-1}$ of $|A_{v,j}|/d$.
\end{lemma}
\begin{proof}
The error bound of $2^{j-1}$ follows from Lemma \ref{lem:APcomb-error}.
We prove the existence of the stated partition by induction on $j$.

For $j=1$, since $p_1=2$ by assumption, we partition $I$ into $A_{1,1}$, the even integers in $I$, and $A_{2,1}$, the odd integers. Each is a single arithmetic progression, hence a $1$-AP combination.

For $j=2$, let $p=p_2$ be the second prime. We construct the four sets as follows. Let $x$ be chosen so that the number of odd integers in $I\cap[1,x]$ that are not divisible by $p$ equals $L/ 4$. To see that such an $x$ exists, note that the majority of odd integers in $I$ are not divisible by $p$.
Indeed, the number of odd elements of $I$ divisible by $p$ is $\le L/(2p)+1$, and this is $< L/4$ by the assumption $L\ge4^j$. Set
\begin{align*}
A_{2p,2} &= \{\text{odd } i\in I : i\le x,\, p\nmid i\},\\
A_{2,2} &= \{\text{odd } i\in I : p\mid i,\, i\le x\}\cup\{\text{odd } i\in I : i>x\}.
\end{align*}
Similarly, let $y$ be chosen so that the number of even integers in $I\cap[1,y]$ not divisible by $p$ equals $L/4$, and set
\begin{align*}
A_{p,2} &= \{\text{even } i\in I : i\le y,\, p\nmid i\},\\
A_{1,2} &= \{\text{even } i\in I : p\mid i,\, i\le y\}\cup\{\text{even } i\in I : i>y\}.
\end{align*}
Each of $A_{2p,2}$ and $A_{p,2}$ is the difference of two arithmetic progressions, hence a $2$-AP combination by rule (3). Each of $A_{2,2}$ and $A_{1,2}$ is the disjoint union of two arithmetic progressions, hence a $2$-AP combination by rule (2). Thus each set is a $2$-AP combination.

For $j\ge 3$, assume the result holds for $j-1$: each $A_{v,j-1}$ is a $2^{j-2}$-AP combination. For each $v\mid m_{j-1}$, we partition $A_{v,j-1}$ into $A_{v,j}$ and $A_{p_jv,j}$ using a cutoff as follows. Let $z_v$ be chosen so that the number of elements of $A_{v,j-1}$ that are $\le z_v$ and not divisible by $p_j$ equals $L/2^j$. Set
\begin{align*}
A_{p_jv,j} &= \{i\in A_{v,j-1}: i\le z_v,\, p_j\nmid i\},\\
A_{v,j} &= \{i\in A_{v,j-1}: p_j\mid i,\, i\le z_v\}\cup\{i\in A_{v,j-1}: i>z_v\}.
\end{align*}
Every element of $A_{p_jv,j}$ is coprime to $p_j$ (and was already coprime to $v$), so is coprime to $p_jv$. Every element of $A_{v,j}$ was already coprime to $v$.

We now verify the sizes. By the induction hypothesis, the number of $p_j$-multiples in $A_{v,j-1}$ is within $2^{j-2}$ of $|A_{v,j-1}|/p_j=L/(p_j2^{j-1})$, hence at most
\[
\frac{L}{p_j2^{j-1}}+2^{j-2}\le \frac{L}{ 3\cdot2^{j-1}}+\frac{L}{ 2^{j+2}}<\frac{L}{2^j},
\]
using $p_j\ge 3$ and $L\ge 4^{j}$. Thus there are enough non-$p_j$-multiples in $A_{v,j-1}$ to fill
$A_{p_jv,j}$ to size $L/2^j$, and $A_{v,j}$ receives the remaining $L/2^j$ elements.

It remains to show each new set is a $2^{j-1}$-AP combination. By induction, we know that each set $A_{v,j-1}$ is a $2^{j-2}$-AP combination.

A key observation is that intersecting with $\{i:i\le z_v\}$ or $\{i:i>z_v\}$ preserves the AP-combination structure of a set. If $T$ is a $k$-AP combination, then $\{i\in T:i\le z_v\}$ and $\{i\in T:i>z_v\}$ are each $k$-AP combinations. (We simply truncate each constituent arithmetic progression. This could result in some of them being empty.) Similarly, restricting to $p_j$-multiples preserves the AP-combination structure, since if $T$ is a $k$-AP combination, then $\{i\in T:p_j\mid i\}$ is a $k$-AP combination (just replace each arithmetic progression by the sub-arithmetic progression of its $p_j$-multiples).

Now consider each of the sets used to construct $A_{v,j}$ and $A_{p_jv,j}$.
First, $\{i\in A_{v,j-1}:p_j\mid i,\,i\le z_v\}$ is just $A_{v,j-1}$ restricted to $p_j$-multiples and then to $\{i\le z_v\}$. So, it has the same combination structure as $A_{v,j-1}$, hence is a $2^{j-2}$-AP combination.
Similarly $\{i\in A_{v,j-1}:i>z_v\}$ is $A_{v,j-1}$ restricted to $\{i>z_v\}$, hence also a $2^{j-2}$-AP combination. Since $A_{v,j}$ is the disjoint union of these two it is a $2^{j-1}$-AP combination.

For $A_{p_jv,j}$, we note that
\begin{align*}
A_{p_jv,j}  &= \{i\in A_{v,j-1}:p_j\nmid i,\,i\le z_v\} \\
&= \{i\in A_{v,j-1}:i\le z_v\}\setminus\{i\in A_{v,j-1}:p_j\mid i,\,i\le z_v\}.
\end{align*}
As argued above, each of these sets is a $2^{j-2}$-AP combination, and the latter is a subset of the former, so their difference is a $2^{j-1}$-AP combination.
This completes the induction.
\end{proof}

\begin{proof}[Proof of Theorem \ref{th:sqfr}]
 Let $u$ be a squarefree number satisfying the hypotheses with $\ell=\omega(u)$. If $u$ is odd, then $2u$ is even with $\omega(2u)=\ell+1\ge46$ prime factors. If we can show that $2u$ is matchable, then $u$ is matchable by Lemma \ref{lem:half}. Thus we may assume $u$ is even, so that $2=p_1<p_2<\cdots<p_\ell$ where $u=p_1p_2\cdots p_\ell$.

 For an integer $j\le\omega(u)/2$ to be chosen later, we let $m_j=p_1p_2\cdots p_j$ and $n = u/m_j$, so that $m_j$ is the product of the $j$ smallest primes dividing $u$ while $n$ contains the rest. We then apply Lemma \ref{lem:partition} using these $j$ smallest prime factors and
$I=[1,\tau(u)]=[1,2^{j+\omega(n)}]$. Since $4^{j} \le \tau(u)$, the lemma allows us to produce a partition of $I$ into sets $A_{v,j}$ as described in the lemma, each having size $\tau(n)$.

 Thus it now suffices to show that there are one-to-one correspondences between $D(n)$ and each of the sets $A_{v,j}$
with corresponding numbers relatively prime. Indeed, for $v\midd m_j$, the correspondence can instead be between $A_{v,j}$ and $vD(n)$, keeping the coprime property. Then, as $D(u)$ is the disjoint union of the sets $vD(n)$ as
$v$ runs over all of the divisors of $m_j$, we can piece together
these matchings and so have a coprime matching of all divisors of $u=m_jn$ to $[1,\tau(u)]$.

 To show the existence of these one-to-one correspondences we note that any divisor $d\mid n$ is coprime to $m_j$, so by Lemma \ref{lem:partition}, the number of integers in $A_{v,j}$ divisible by
$d$ is within $2^{j-1}$ of $\tau(n)/d$.

We choose $j$ as follows. For $\ell \ge68$ we take $j=\lfloor\sqrt{\ell}\rfloor$ and for
$45\le\ell\le67$
we take $j=4$ except $j=3$ when $\ell\in\{46,47,48\}$ and $j=5$ when $\ell=52$. Note that in every case we have $j\le\sqrt\ell$.
With these choices, $\omega(n) = \ell - j \geq 41$ and one can verify that
\begin{equation}
f(n) :=\sum_{p\,|\,n}\frac1p= \sum_{i=j+1}^{\ell} \frac{1}{p_i} \leq \sum_{i=j+1}^{\ell} \frac{1}{P_i} < \frac{93}{100},
\label{eq:primesumbound}
\end{equation}
where $P_i$ is the $i$-th prime. To see this in the case that $j=\lfloor\sqrt\ell\rfloor$, we use that $P_i>i\log i$ (see \cite{R}), so 
\[
f(n)<\frac1{\sqrt{68}\log\sqrt{68}}+\int_{\sqrt\ell}^\ell\frac{dt}{t\log t}<0.06+\log2<0.76.
\]

We use Hall's theorem for the bipartite graph on $A=A_{v,j}$ to $D(n)$ where there is an edge precisely when $a\in A$ and $d\midd n$ are coprime.
Suppose that $S\subset A$ and that $s\in S$ minimizes $k=\omega((s,n))$.
We wish to show that $|S|$ is bounded above by the size of the neighborhood $N(S)$ of $S$. If $k=0$ then the neighborhood of $S$ is all of $D(n)$, so the condition holds.

Assume that $k\ge1$.  Since each element $s$ of $S$ is in $A\subset[1,\tau(u)]$ and is
divisible by some $d\mid n$ with $\omega(d)=k$, we have
\[
|S|\le\sum_{\substack{d\,|\,n\\\omega(d)=k\\d\leq \tau(u)}}\sum_{\substack{a\in A\\d\,|\,a}}1
\le\sum_{\substack{d\,|\,n\\\omega(d)=k}}\left(\frac{\tau(n)}d+2^{j-1}\right).
\]
By the multinomial theorem, we have
\[
\sum_{\substack{d\,|\,n\\\omega(d)=k}}\frac1d\le\frac{f(n)^k}{k!},
\]
so that
\begin{align} |S| &\le\tau(n)\frac{f(n)^k}{k!}+2^{j-1}\binom{\omega(n)}{k} \notag \\
&< \frac{\tau(n)}{2^{k}}\left(\frac{(93/50)^k}{k!}+2^{2j-1+k-\ell + \log_2{\binom{\omega(n)}{k}}}\right).
\label{eq:keq2Sbound}
\end{align}
Since the primes dividing $n$ are all at least $P_{j+1}$, a divisor $d$ of $n$ with $\omega(d)=k$ satisfies $d \geq P_{j+1}P_{j+2}\cdots P_{j+k}$. For such a $d$ to divide any element of
$A\subset [1,\tau(u)]$, we need $d \leq \tau(u) = 2^\ell$, as we have seen, so we only need consider $k \leq
\overline{k}$, where $\overline{k}$ is the largest integer with $P_{j+1} P_{j+2} \cdots P_{j+
\overline{k}} \leq 2^\ell$.

\subsection{Case $k\ge4$.} For Hall's condition to hold, since the neighborhood of $S$ has size at least $\tau(n)/2^k$, it suffices if
\begin{equation}\label{eq:kge4constraint}
\frac{(93/50)^k}{k!}+2^{2j-1+k-\ell + \log_2\binom{\omega(n)}{k}} < 1
\end{equation}
for each $k$ with $4 \leq k \leq \overline{k}$. Note that $(93/50)^4/24 < 0.50$, so we need only
show that the second term is at most $\frac12$.

Let $E_k:= 2j-1+k-\ell+\log_2\binom{\omega(n)}k$ denote the exponent, so the second term equals $2^{E_k}$.

\textit{Large $\ell$ ($\ell \geq 192$).} Set $j = \lfloor\sqrt{\ell}\rfloor$. We first bound $\overline{k}
$. Since $j = \lfloor\sqrt{\ell}\rfloor \geq 13$ for $\ell\ge192$, we have $P_{j+1} \geq P_{14} = 43$. The constraint $\prod_{i=1}^{k} P_{j+i} \leq 2^\ell$ implies $43^k < 2^\ell$, giving $\overline{k}
<\ell/\log_243 < 0.185\ell$. With $\omega(n)= \ell - j > \ell - \sqrt{\ell}$, we have $\overline{k}/
\omega(n)< 0.20$ for $\ell \geq 192$.

Using the entropy function $H(x)=-x\log_2 x-(1-x)\log_2(1-x)$ and the standard bound $\log_2\binom{a}{b}<a\cdot H(b/a)$, writing $c=k/\omega(n)<0.20$ and using that $H$ is increasing on $(0,\tfrac{1}{2})$, we have $H(c)<H(0.20)<0.722$, so $c+H(c)<0.922$.
Thus
\begin{align*}
E_k&\le 2j-1+(c+H(c))\omega(n)-\ell\\ &<2\sqrt{\ell}-1+0.922(\ell-\sqrt{\ell})-\ell\\ &=-0.078\ell+1.078\sqrt{\ell}-1<-1.03
\end{align*}
for $\ell\ge192$. Since $(93/50)^4/4!<0.50$ and $2^{E_k}<0.50$, we have
\eqref{eq:kge4constraint} for $4\le k\le\overline{k}$.

\textit{Moderate $\ell$} ($68 \leq \ell < 192$). We use $j = \lfloor\sqrt{\ell}\rfloor$ and verify
\eqref{eq:kge4constraint} by direct computation. For each $\ell$ in this range, we find \[\overline{k} = \max\{k:\prod_{i=j+1}^{j+k} P_i <2^\ell\}.\] Then, for each $k$ with $4 \leq k \leq
\overline{k}$, we compute $E_k = k + 2j-1 - \ell + \log_2\binom{\ell-j}{k}$ and verify that
$(93/50)^k/k! + 2^{E_k} < 1$.

\textit{Small $\ell$} ($45 \leq \ell \leq 67$). Here we take (as above) $j=4$ when $\ell=45$, $j=3$ for $46\le\ell\le48$, $j=4$ for $49\le\ell\le67$ (except $j=5$ when $\ell=52$). As above, we verify
\eqref{eq:kge4constraint} directly for all $4\leq k \leq \overline{k}$. This concludes the case
$k\ge4$ for all $\ell\ge45$.

\subsection{Case $k=3$.} Here the neighborhood of $S$ has at least $\tau(n)/8$ elements. Let
$a_1\in S$ with $(a_1,n)=q_1q_2q_3$ for distinct primes $q_1,q_2,q_3$ dividing $ n$.

First, suppose that the triple $\{q_1,q_2,q_3\}$ is unique (every $a\in S$ has $(a,n)=q_1q_2q_3$ or $\omega((a,n))>3$). Then
\begin{align*}
|S| &\leq \sum_{\substack{a\in A\\q_1q_2q_3\mid a}}1 + \sum_{\substack{d\mid n\\\omega(d)=4}}
\left(\frac{\tau(n)}{d}+2^{j-1}\right)\\
&\leq \frac{\tau(n)}{q_1q_2q_3}+2^{j-1} + \tau(n)\frac{f(n)^4}{24} + 2^{j-1}\binom{\omega(n)}{4}\\ &< \frac{\tau(n)}{8}\left(\frac{8}{q_1q_2q_3}+\frac{8f(n)^4}{24}+ 2^{2\sqrt{\ell}+2-\ell}\left(1+
\binom{\ell}{4}\right)\right),
\end{align*}
using that $\tau(n)=2^{\ell-j}$.  Since $q_1q_2q_3 \geq 7\cdot11\cdot13 = 1001$ for $j=3$ and $f(n)<0.93$, the expression in
parentheses is at most $8/1001 + 8(0.93)^4/24 + 2^{2\sqrt{\ell}+2-\ell}\binom{\ell}{4} < 0.26 + 2^{2\sqrt{\ell}+2-\ell}\binom{\ell}{4}$, which is $<1$ when $\ell\ge27$.

Now suppose $S$ contains at least 2 elements $a_1,a_2$ with $\omega((a_i,n))=3$, say $ (a_1,n)=q_1q_2q_3$ and $(a_2,n)=q_4q_5q_6$, with $\{q_1,q_2,q_3\}\ne\{q_4,q_5,q_6\}$. The two triples share at most 2 primes, so by inclusion-exclusion $|N(S)|\ge 2\cdot\tau(n)/8 - \tau(n)/16
= 3\tau(n)/16$. Also,
\begin{align*}
|S| &\leq \tau(n)\frac{f(n)^3}{6} + 2^{j-1}\binom{\omega(n)}{3}\\
&< \frac{3\tau(n)}{16}\left(\frac{16f(n)^3}{18} + \frac{16}{3}\cdot 2^{2\sqrt{\ell}-1-\ell}\binom{\ell}
{3}\right).
\end{align*}
Since $f(n)<0.93$, we have $16(0.93)^3/18 < 0.72$, and $\frac{16}{3}\cdot 2^{2\sqrt{\ell}-1-\ell}
\binom{\ell}{3}<0.28$ for $\ell\ge 25$. This completes the case $k=3$.

\subsection{Case $k=2$.} Here the neighborhood of $S$ has at least $\tau(n)/4$ elements. Let
$a_1 \in S$ with $(a_1,n)=q_1q_2$ for distinct primes $q_1,q_2 \mid n$.

First, suppose that the pair $\{q_1,q_2\}$ is unique (every $a\in S$ has
$(a,n)=q_1q_2$ or $\omega((a,n))>2$). Then
\begin{align*}
|S| &\leq \sum_{\substack{a\in A\\q_1q_2\mid a}}1 + \sum_{\substack{d\mid n\\\omega(d)=3}}
\left(\frac{\tau(n)}{d}+2^{j-1}\right)
\leq \frac{\tau(n)}{q_1q_2}+2^{j-1} + \tau(n)\frac{f(n)^3}{6} + 2^{j-1}\binom{\omega(n)}{3}\\
&< \frac{\tau(n)}{4}\left(\frac{4}{q_1q_2}+\frac{4f(n)^3}{6}+ 2^{2\sqrt{\ell}+1-\ell}\left(1+\binom{\ell}
{3}\right)\right).
\end{align*}
Since $q_1q_2 \geq 77$ for $j \geq 3$ and $f(n)<0.93$, we have $4/77 + 4(0.93)^3/6 < 0.59$, and
$2^{2\sqrt{\ell}+1-\ell}(1+\binom{\ell}{3})<0.33$ for $\ell \geq 23$.

Now suppose $S$ contains 2 elements $a_1,a_2$ with $\omega((a_i,n))=2$, say $(a_1,n)=q_1q_2$ and $(a_2,n)=q_3q_4$, with $\{q_1,q_2\}\ne \{q_3,q_4\}$.
Assume also that every other $a\in S$ has either $(a,n)=q_1q_2$, $(a,n)=q_3q_4$ or $\omega((a,n))\ge3$.
 The pairs share at most one prime, so by inclusion-exclusion $|N(S)|\ge 2\cdot\tau(n)/4 - \tau(n)/ 8 = 3\tau(n)/8$.

Also,
\begin{align*}
|S| &\leq \sum_{\substack{a\in A\\q_1q_2\mid a}}1 + \sum_{\substack{a\in A\\q_3q_4\mid a}}1 +
\sum_{\substack{d\mid n\\\omega(d)\ge 3}}\left(\frac{\tau(n)}{d}+2^{j-1}\right)\\
&< \frac{3\tau(n)}{8}\left(\frac{8}{3}\cdot\frac{2}{77}+\frac{8f(n)^3}{18}+ \frac{8}{3}\cdot 2^{2\sqrt{\ell}-1-\ell}\left(2+\binom{\ell}{3}\right)\right).
\end{align*}
Since $f(n)<0.93$, we have $16/(3\cdot 77) + 8(0.93)^3/18 < 0.43$, and $\frac{8}{3}\cdot 2^{2\sqrt{\ell}-1-\ell}(2+\binom{\ell}{3})<0.5$ for $\ell\ge21$.

So now assume that there are at least 3 different values of $(a,n)$ for $a\in S$ with exactly 2 prime factors. If the 3 values are $q_1q_2,q_3q_4,q_5q_6$, then the case when one prime is shared among all 3 numbers gives the smallest size for $N(S)$ and that size is $(7/16)\tau(n)$. Then
\[
|S|<\frac{7\tau(n)}{16}\left(\frac{16f(n)^2}{14}+\frac{16}7\cdot2^{2\sqrt{\ell}-1-\ell}\binom{\omega(n)} 2\right)<\frac{7\tau(n)}{16}(0.99+0.01),
\]
for $\ell\ge26$, completing the case $k=2$.

\subsection{Case $k=1$.} Here the neighborhood of $S$ has at least
 $\tau(n)/2$ elements, so we may assume that $|S| \geq \tau(n)/2$. Let $a_1 \in S$ with $ (a_1,n)=q_1$ for some prime $q_1 | n$.

 First, suppose that $q_1$ is unique (every $a\in S$ has $(a,n)=q_1$ or $\omega((a,n))\ge2$). Then
\begin{align*}
|S|&\le \sum_{\substack{a\in A\\q_1\mid a}}1+
\sum_{\substack{a\in A,\,q_1\nmid a\\\omega((a,n))\ge2}}1
\le \frac{\tau(n)}{q_1}+2^{j-1}+\sum_{\substack{d\mid n,\,q_1\nmid d\\\omega(d)=2}}
\left(\frac{\tau(n)}d+2^{j-1} \right)\\
&\le\frac{\tau(n)}{q_1}+2^{j-1}+\tau(n)\frac{(f(n)-1/q_1)^2}2+2^{j-1}\binom{\omega(n)-1}2\\ &<\tau(n)\left(\frac{f(n)^2}{2}+\frac{1-f(n)}{q_1}+\frac{1}{2q_1^2}\right)+2^{j-1}\left(1+
\binom{\omega(n)-1}2\right).
\end{align*}
Using $f(n) < 93/100$ and $q_1\geq P_{j+1} \geq 7$ for $j\ge3$, the coefficient of $\tau(n)$ is at most
$0.432 + 0.07/7 + 0.5/49 < 0.453$. Since $\tau(n)=2^{\ell-j}$,
\begin{align*}
|S|&<0.453\tau(n)+2^{j-1}\left(1+\binom{\omega(n)-1}2\right)\\ &\le \frac{\tau(n)}2\left(0.906+2^{2\sqrt{\ell}-\ell}\left(1+\binom{\ell-4}2\right)\right).
\end{align*}
This is $<\tau(n)/2$ when $\ell\ge30$.

Now suppose $S$ contains elements $a_i$ with $(a_i,n)=q_i$ for $r\ge2$ distinct
primes $q_1,\dots,q_r$ dividing $n$. The neighborhood of $S$ contains all divisors of $n$ coprime to at least one of $q_1,\dots,q_r$.

If $r=2$, this neighborhood has cardinality $\frac34\tau(n)$. We have
\begin{align*}
|S|&\le\sum_{i=1}^2\left(\frac{\tau(n)}{q_i}+2^{j-1}\right)+\sum_{\substack{d\midd n\\\omega(d)=2}}
\left(\frac{\tau(n)}d+2^{j-1}\right)\\
&<\tau(n)\left(\frac1{q_1}+\frac1{q_2}+\frac{f(n)^2}2\right)+2^{j-1}\left(2+\binom{\omega(n)}2\right).
\end{align*}
Since $q_1,q_2\ge P_{j+1}$ and $f(n)<93/100$, the first term is at most $2/7 + 0.433 < 0.72\tau(n)
$ for $j\ge3$. Thus,
\[
|S|<\frac{3\tau(n)}4\left(0.96+\frac{4}3\cdot 2^{2\sqrt{\ell}-1-\ell}\left(2+\binom{\ell-3}2\right)\right).
\]
This is $<\frac34\tau(n)$ when $\ell\ge22$.

If $r=3$, the neighborhood has cardinality at least $\frac78\tau(n)$, and
\begin{align*}
|S|&<\tau(n)\left(\frac1{q_1}+\frac1{q_2}+\frac1{q_3}+\frac{f(n)^2}2\right)+2^{j-1}\left(3+
\binom{\omega(n)}2\right)\\
 &<\frac{7\tau(n)}8\left(\frac{0.862}{7/8}+\frac{8}7\cdot 2^{2\sqrt{\ell}-1-\ell}\left(3+\binom{\ell-3} 2\right)\right),
\end{align*}
 using that $q_1,q_2,q_3$ are distinct primes $\ge7$ for $j\ge3$ and $f(n)<93/100$. Our estimate for $|S|$ is $<\frac78\tau(n)$ for $\ell\ge35$.

If $r\ge4$, the neighborhood has cardinality at least $\frac{15}{16}\tau(n)$, and
\begin{align*}
|S|&\le\sum_{\substack{a\in A\\\omega((a,n))\ge1}}1\le\sum_{q\mid n}\sum_{\substack{a\in A\\q\midd a}}1
\le\tau(n)f(n)+2^{j-1}\omega(n)\\
&<\frac{93}{100}\tau(n)+2^{j-1}\omega(n)<\frac{15\tau(n)}{16}\left(\frac{93/100}{15/16}+\frac{16}
{15}\cdot
2^{2\sqrt{\ell}-1-\ell}\ell\right).
\end{align*}
This is $<\frac{15}{16}\tau(n)$ for $\ell\ge13$. Hall's condition holds, completing the proof.
\end{proof}

\section{Squarefree numbers with few prime factors} In this section we prove the following theorem.
\begin{theorem}
\label{th-few}
Every squarefree number with at most $44$ prime factors is matchable.
\end{theorem}

 We illustrate the argument for $\ell=24$, the smallest case that exhibits the full complexity, and then discuss the modifications for other $\ell\le44$. We first establish Proposition \ref{prop:oddmatch} at the end of this section for $\ell=24$: we show there is a coprime matching from $D(u)$ to the odd integers in $[1,2^{25}]$, where $u$ is odd, squarefree, and with 24 prime factors.

Let $u = q_1q_2\cdots q_{24}$ be the product of any $24$ distinct
odd primes $3 \leq q_1<q_2<\cdots<q_{24}$, so $\tau(u)=2^{24}=16{,}777{,}216$. We wish to show there is a coprime matching between the $2^{24}$ odd integers in $[1,2^{25}]$ and $D(u)$.
By Hall's theorem it suffices to show that for every set $S$ of odd integers in $[1,2^{25}]$, the neighborhood
\[
N(S) \;=\; \{\,d \mid u : \gcd(d,a)=1 \text{ for some } a\in S\,\}
\]
satisfies $|N(S)|\ge|S|$.

Let $c_i(u)$ denote the number of odd integers in $[1,2^{25}]$ sharing exactly $i$ prime factors with $u$, \[c_i(u) \coloneqq \#\{\text{odd }a\in[1,2^{25}]: \omega((a,u))=i\}\] and $c_{\geq i}(u) =
\sum_{i'\geq i}c_{i'}(u)$ the count of those odd integers sharing at least $i$ prime factors with $u$.  A key observation is that the counts $c_{\ge i}(u)$ are only made larger if the odd
primes comprising $u$ are replaced with smaller odd primes.

\begin{lemma}\label{lem:monotonicity}
Let $u$ be a squarefree odd number as above and let $p,q$ be odd primes with
$q\mid u$,  $p\nmid u$, and $p<q$. Then
$c_{\ge i}(u)\le c_{\ge i}(u\cdot p/q)$ for all $i\ge1$.
\end{lemma}
\begin{proof}
Set $v=u/q$, so $u=vq$ and $u\cdot p/q=vp$.
We construct an injection $f$ from integers counted by $c_{\ge i}(vq)$ to those counted by $c_{\ge i}(vp)$.
For each integer $a$ counted by $c_{\ge i}(vq)$, set
\[
f(a)=\begin{cases}
a & \text{if } \omega((v,a))\ge i, \text{ or } pq\mid a \text{ and } \omega((v,a))=i-1, \\
a\cdot p^j/q^j & \text{otherwise (where } q^j\,\|\, a\text{)}.
\end{cases}
\]
Note that if $a$ falls into the second case above, we must have $\omega((v,a))=i-1$ and $p\nmid a$.
In each of the conditions for the first case above we find that $a=f(a)$ is counted by both $c_{\ge i}(vq)$ and $c_{\ge i}(vp)$. Note that in the second case above, since $p<q$ we have $f(a)<a\le 2^{25}$ and $f(a)$ odd, and since $p,q\nmid v$ the primes of $v$ dividing $f(a)$ are the same as those dividing $a$, so $\omega((vp,f(a)))=(i-1)+1=i$.  Thus $f$ maps into the target set.

For injectivity note that in the second case, since $q\nmid f(a)$, $\omega((vq,f(a)))=i-1<i$, so the images $f(a)$ lie outside the domain and cannot collide with the images from the first-case.
\end{proof}

Let $n = 3\cdot5\cdot7\cdots97$ denote the product of the \emph{smallest} $24$ odd primes, so $q_i \ge p_i$ where $p_i$ is the $i$-th odd prime. Applying Lemma \ref{lem:monotonicity} repeatedly, each time replacing a prime of $u$ with a smaller prime of $n$, gives $c_{\ge i}(u)\le c_{\ge i}(n)$ for each $i\ge1$.

\medskip
The rest of the proof will rely heavily on the counts $c_i(n)$ and $c_{\geq i}(n)$, which we will write as $c_i$ and $c_{\ge i}$, respectively. We will also occasionally need counts for the number of integers having a fixed greatest common divisor with $u$ (respectively $n$).

As with the counts $c_{\ge i}$, the number of odd integers in the interval $[1,2^{25}]$ whose gcd with
$u$ is $q_{j_1}q_{j_2}\cdots q_{j_i}$ is majorized by the number of such integers whose gcd with
$n$ is $p_{j_1}p_{j_2}\cdots p_{j_i}$. We denote $\gcd_d \coloneqq \#\{\text{odd }a\in[1,2^{25}]: (a,n)=d\}$.

We record in Table \ref{tab:oddcensus} the computed values of $c_i$ for $3\leq \ell \leq 44$, and in Table \ref{tab:oddgcds} the computed values of $\gcd_d$ for $d=105, 15, 21, 3, 5$ (all computations\footnote{Python code used to generate the numbers in these tables is included with the arXiv version of this paper.} were performed using exact values; large entries in the tables are rounded up, as noted in the captions). The former table also contains a column, $\omega_{\max}$ containing the largest value of $i$ such that $c_i$ is nonzero, and the latter table includes a column $x_3$, containing the count of odd integers in $[1,2^{25}]$ which are either divisible by 3 or counted by
$c_{\ge 3}$, whose purpose will be explained shortly. In the latter table, values are only included when needed in the analogue of the argument described below (blank entries are not necessarily 0).

We will use frequently the observation that if $k\coloneqq \min_{a\in S}\omega((a,u))$ then
\[
|N(S)| \;\ge\; 2^{24-k} \;=\; \frac{\tau(u)}{2^k}.
\]

\medskip
\noindent Let $S$ be a nonempty set of odd integers in $[1,2^{25}]$; we verify Hall's condition by considering the possible values of $k = \min_{s\in S}\omega((s,u))$.

By summing the values of $c_i$ in Table \ref{tab:oddcensus}, in the row for $\ell = 24$ we find that
\[
c_{\ge 5}=88{,}525,\quad c_{\ge 4}=485{,}129,\quad c_{\ge 3}=2{,}151{,}882,\quad
c_{\ge 2}=6{,}377{,}708,\quad
c_{\ge 1}=12{,}741{,}251.
\]

\subsection{Step 1: Cases $k\ge4$.}
First, suppose $k\geq 5$. Then by the monotonicity described above we have $|S|\leq c_{\ge 5} = 88{,}525$. On the other hand, since $\omega_{\max}=7$ in the row $\ell=24$ of Table \ref{tab:oddcensus}, we have $\omega((s,u))\leq 7$ for all $s$, and thus $|N(S)|\geq |N(s)|
\geq 2^{24-7} = 131{,}072$. So $|S|<|N(S)|$ and Hall's condition holds.

If $k=4$ then every element $s \in S$ has $\omega((s,u))\ge4$ so
$|S|\le c_{\ge 4}=485{,}129$. Since at least one element has $\omega((s,u))=4$ we find that
$|N(S)|\ge 2^{24-4}=1{,}048{,}576>485{,}129\ge|S|$,
so again, Hall's condition holds.

\subsection{Step 2: Case $k=3$ }

Every element of $S$ has $\omega\ge3$, so $|S|\le c_{\ge3}=2{,}151{,}882$, while $|N(S)|\ge2^{21}=2{,}097{,}152$.
Since $2{,}151{,}882>2{,}097{,}152$ we cannot immediately conclude that Hall's condition holds.

Let $s \in S$ be such that $\omega((s,u))=3$ and let $d=(s,u)$. Then $d \geq 105=3\times 5
\times 7$ and, by the monotonicity mentioned above we have $\gcd_d \leq \gcd_{105}=83{,}729$.

If $d$ were unique, then every element of $S$ sharing exactly 3 prime factors with $u$ would need to have greatest common divisor $d$ with $u$. Then we would have \[ |S| \leq \gcd\nolimits_d + c_{\geq 4}(u) \leq \gcd\nolimits_{105} +c_{\geq 4} = 83{,}729+ 485{,}129 <2^{21}\le |N(S)|,\] and so Hall's condition would be satisfied. So we suppose $d$ is not unique, namely there is a second
element, $s'$ having $\omega((s',u))=3$ but $(s,u)\neq (s',u)$. In this case, since $(s,u)$ and $ (s',u)$ can share at most two primes, considering the neighborhood of just $s$ and $s'$ we find, by inclusion-exclusion that \[|N(S)|\geq
|N(s)\cup N(s')| = 2^{21}+2^{21}-2^{20} = \frac{3}{16}\cdot2^{24} = 3{,}145{,}728.
\]
But $|S|\le c_{\ge3}=2{,}151{,}882<3{,}145{,}728\le|N(S)|$, hence Hall's condition holds when
$k=3$.

\subsection{Step 3: Case $k=2$}

Every element of $S$ has $\omega\ge2$, so $|S|\le c_{\ge2}=6{,}377{,}708$, while $|N(S)|\ge2^{22}=4{,}194{,}304$.
Again, we cannot immediately conclude using Hall's theorem, but we can assume $|S|>|N(S)|
\ge2^{22}$.

Let $s \in S$ such that $\omega((s,u))=2$ and let $e=(s,u)$. Then $e \geq 15=3\times 5$ and, by monotonicity, $\gcd_e \leq \gcd_{15}=504{,}881$.
If $e$ were unique, we would find that \[|S|\leq \gcd\nolimits_{15} + c_{\ge 3} = 504{,}881+2{,} 151{,}882 = 2{,}656{,}763
< 2^{22}\le|N(S)|,
\]
satisfying Hall's condition, and so we assume that $e$ is not unique, namely there exists a second
$s'\in S$ with $\omega((s',u))=2$ but $(s',u) = e' \neq e$.
Since $e$ and $e'$ share at most one prime factor, we can update our lower bound for the neighborhood $N(S)$ to
\begin{equation}
|N(S)|\ge|N(s) \cup N(s')|=2^{22} +2^{22}-2^{21} = \frac38\times 2^{24} = 6{,}291{,}456.
\label{eq:twogcdtwo}
\end{equation}

Note that this is still less than $c_{\ge 2}$. Now, if $e$ and $e'$ were the only two such divisors, then
\begin{align*}
|S| &\leq\gcd\nolimits_e + \gcd\nolimits_{e'} + c_{\ge 3} \leq \gcd\nolimits_{15} +
\gcd\nolimits_{21} + c_{\ge 3} \\
&= 504{,}881+336{,}514+2{,}151{,}882 = 2{,}993{,}277
< 6{,}291{,}456\le|N(S)|.
\end{align*}
Again, Hall's condition is satisfied in this case, leaving us with the possibility that there is a third
$s''\in S$, with $\omega((s'',u))=2$ and where $(s'',u)=e'' \notin \{e,e'\}$. In this case the neighborhood of $S$ is minimized in the situation when $e,e',e''$ all share a single prime factor, in which case it is bounded below by 
\[
|N(S)|\ge|N(s) \cup N(s') \cup N(s'')| \ge 3\cdot 2^{22}-3\cdot 2^{21} +2^{20}= 3\cdot2^{21} +
2^{20} = \frac7{16} \cdot2^{24} = 7{,}340{,}032.
\]
But $|S|\le c_{\ge2}=6{,}377{,}708<7{,}340{,}032\le|N(S)|$. Hence, in every case Hall's condition holds when $k=2$.

\subsection{Step 4: Case $k=1$}

Every element of $S$ has $\omega\ge1$, so $|S|\le c_{\ge1}=12{,}741{,}251$,
while $|N(S)|\ge2^{23}=8{,}388{,}608$, so we will again need to work with the specific gcds.

Proceeding as above, we suppose $s \in S$ has $\omega((s,u))=1$, with $(s,u)=q$ and suppose that $q$ is unique. Then
\[|S|\leq \gcd\nolimits_3 + c_{\ge 2} = 2{,}019{,}785+6{,}377{,}708 = 8{,}397{,}493
> 2^{23}.
\]
Note that unlike in previous steps, this bound does not allow us to conclude (yet) that there is another $s'$ and another prime $q'$ distinct from $q$ with $(s',u)=q'$. So we consider again those elements of $S$ sharing two prime factors with $u$, one of which is $q$.

Let $x_q$ denote the total number of odd numbers in $[1,2^{25}]$ which are either divisible by
$q$ or counted by $c_{\ge 3}$. As with other statistics, this count is majorized by the count
$x_3=6{,}334{,}949$ which is included in Table \ref{tab:oddgcds}.

Since this count is smaller than our bound $|N(S)|\ge2^{23}$, we would be done if $S$ consisted only of elements counted by $x_3$. So we suppose that $S$ contains $s'$, not counted by
$x_q$. Then $\omega((s',u))\leq 2$ and $q\nmid s'$. In this case, we can update our lower bound on $|N(S)|$. This quantity is smallest when $s'$ shares precisely two prime factors with
$u$, neither of which is $q$, in which case we find that

\[
|N(S)|\ge|N(s) \cup N(s')| = 2^{23} +2^{22}- 2^{21}= 5\cdot 2^{21} =  \frac58\cdot2^{24} = 10{,}
485{,}760.
\]
Since this quantity now exceeds $\gcd\nolimits_3 + c_{\ge 2}$, we find that Hall's criterion is necessarily satisfied unless $(s',u)=q'$ where $q'\neq q$ is a different prime factor. Now our lower bound for $|N(S)|$ improves to
\[
|N(S)| \geq |N(s)\cup N(s')| = 2^{23}+2^{23}-2^{22}
= 3\cdot2^{22} = \frac34\cdot2^{24} = 12{,}582{,}912.
\]

Using this bound, we now return to our original line of argumentation. This bound exceeds
\begin{align*}
c_{\geq 2} + \gcd\nolimits_q + \gcd\nolimits_{q'} &\leq c_{\geq 2} + \gcd\nolimits_3 +
\gcd\nolimits_{5}\\
&= 6{,}377{,}708+2{,}019{,}785+1{,}010{,}179 = 9{,}407{,}672 < 12{,}582{,}912.
\end{align*}
From this, we see that Hall's condition will be satisfied unless $S$ contains a third element $s''$ with $\omega((s'',u))=1$ and
$(s'',u)=q''$ for some prime divisor $q''\ne q,q'$ of $u$.

With three elements $s$, $s'$, and $s''$ all having mutually distinct prime gcds with $u$, inclusion-exclusion gives
\[
|N(S)|\ge |N(s) \cup N(s') \cup N(s'')| \ge 3\cdot2^{23}-3\cdot2^{22}+2^{21}
= \frac78 \cdot2^{24} = 14{,}680{,}064.
\]
But $|S|\le c_{\ge1}=12{,}741{,}251<14{,}680{,}064\le|N(S)|$, so Hall's condition holds when $k=1$.

\subsection{Step 5: Case $k=0$.}

If the minimum value of $\omega((s,u))$ over $s\in S$ is $0$, then some
$s\in S$ is coprime to $u$, so $N(S)=D(u)$ and $|N(S)|=2^{24}\ge|S|$ and Hall's condition holds immediately.

In all cases we have $|N(S)|\ge|S|$, so Hall's condition holds and a perfect matching exists. Since the argument used only upper bounds on census counts, and these upper bounds hold for any $u=q_1\cdots q_{24}$ with $q_i\ge p_i$, the result applies to every squarefree product of $24$ distinct odd primes.

\subsection{Adjustments for other $\ell \leq 44$}

Essentially the same argument can be carried through using any $\ell$ in place of 24, for $3\leq
\ell \leq 44$. The only important adjustment is to change the computed $c_{\geq j}$ values and $
\gcd_d$ values according to the entries in Table \ref{tab:oddcensus} and Table \ref{tab:oddgcds}.

For values of $\ell<24$ some of the steps above can be omitted, for example when $\ell<23$, the argument involving $\gcd_{21}$ is unnecessary. After considering $c_{\geq 3} + \gcd_{15}$ and using it to conclude that there exists $s' \in S$ with $\omega((s',u))\leq 2$, $(s',u)\neq (s,u)$ it is already possible to conclude directly that $|N(S)|>|S|$, since in this case we find, as in
\eqref{eq:twogcdtwo} that $|N(S)|\geq \frac{3}{8}\cdot 2^\ell$, which is already greater than
$c_{\ge 2}$. When a quantity is unneeded for the argument, it is omitted from Table
\ref{tab:oddgcds}.

For values of $\ell > 24$, sometimes additional steps are necessary in the initial Step 1. For example when $\ell=40$, the maximum value of $\omega((s,u))$ over $s\in S$ is 10 (as noted in Table \ref{tab:oddcensus} in the $\omega_{\max}$ column). Since $2^{40-10}=1{,}073{,}741{,} 824\geq c_{\ge 7}$, Hall's condition is satisfied for any $k\ge 7$. One can then check for each
$4\leq k <7$ that $2^{40-k}\geq c_{\ge k}$ so the condition is met for each of these $k$ as well. Putting this all together, we have shown the following.
\begin{proposition} \label{prop:oddmatch}
  For any $\ell \le 44$ and $u$ the product of $\ell$ distinct odd primes, there exists a coprime matching between $D(u)$ and the odd integers in $[1,2^{\ell+1}]$.
\end{proposition}

An identical Hall's theorem argument establishes Proposition \ref{prop:intervalmatch}, with the bipartite graph now having $D(u)$ matched to all integers in $[1,2^\ell]$ (rather than just odd integers in $[1,2^{\ell+1}]$), and using the values in Table \ref{tab:censusinterval} and Table \ref{tab:gcdsinterval} in place of Tables \ref{tab:oddcensus} and \ref{tab:oddgcds}.

\begin{proposition} \label{prop:intervalmatch}
  For any $\ell \le 44$ and $u$ the product of $\ell$ distinct odd primes, there exists a coprime matching between $D(u)$ and the integers in $[1,2^{\ell}]$.
\end{proposition}

We can now combine these propositions to give a proof of Theorem \ref{th-few}.

\begin{proof}[Proof of Theorem \ref{th-few}]
  If $u$ is squarefree, odd, and has at most 44 prime factors, then $u$ is matchable by Proposition \ref{prop:intervalmatch}. If $u$ is even and has $\ell \leq 44$ prime factors, write
$u=2u'$; we create a matching as follows. By Proposition \ref{prop:oddmatch} applied to $u'$ (which has $\ell-1$ odd prime factors), there exists a coprime matching of the divisors of $u'$ to the odd integers in $[1,2^\ell]$. We associate to each even divisor $2d$ of $u$ the odd integer matched to $d$ by this proposition. Then, by Proposition \ref{prop:intervalmatch} applied to $u'$, there exists a coprime matching of $D(u')$ to the integers in $[1,2^{\ell-1}]$; for every odd divisor
$d$ of $u$ (which is also a divisor of $u'$), we match $d$ to $2a$, where $a\in[1,2^{\ell-1}]$ is the integer matched to $d$.

  The first construction matches the $2^{\ell-1}$ even divisors of $u$ bijectively to the $2^{\ell-1}$ odd integers in $[1,2^\ell]$, and the second matches the $2^{\ell-1}$ odd divisors bijectively to the
$2^{\ell-1}$ even integers in $[2,2^\ell]$; together they give a perfect matching from $D(u)$ to
$[1,2^\ell]=[1,\tau(u)]$. Coprimality is preserved: $(2d,a)=(d,a)=1$ since $a$ is odd, and
$(d,2a)=(d,a)=1$ since $d$ is odd.
\end{proof}

\section{M-numbers}

Say a positive integer is an M-number if it is not divisible by any $p^p$ for $p$ prime, i.e., $v_p(n)\le p-1$ for all primes $p$. Call a prime $p\mid n$ \emph{tight} if
$v_p(n)=p-1$, and write $n=n_Tn_R$ where $n_T=\prod_{p|n, p\text{ tight}}p^{p-1}$ and
$n_R=n/n_T$. Set $r=\mathrm{rad}(n_T)=\tau(n_T)$. Note that $(n_T,n_R)=1$: any prime $p$ dividing both would satisfy $v_p(n)\ge p$, contradicting the M-number condition.

We conjecture that every M-number is matchable. In this section we explain how to generalize the proofs of Lemma \ref{lem:partition} and Theorem \ref{th:sqfr} to show that every
M-number with sufficiently many non-tight prime factors and not divisible by the square of any large non-tight prime is matchable (Theorem \ref{thm:M}),
which implies in particular that the set of non-matchable M-numbers has asymptotic density $0$ (Corollary \ref{cor:M}).

The key tool is a partition lemma analogous to Lemma \ref{lem:partition}, but adapted to handle the tight and non-tight primes separately. (Note that in Lemma \ref{lem:partition} the only tight prime is 2, which is why 2 is handled separately.) Let $p_1<p_2<\cdots<p_\ell$ denote the non-tight prime factors of $n_R$ in increasing order, where $\ell=\omega(n_R)$, and set $a_i=v_{p_i}(n)$ for each $i$. For a parameter $j$ with $0\le j\le\ell$, let $m_j=\prod_{i\le j}p_i^{a_i}$, set $n'=n_R/m_j$, and $K=2^j$.

\begin{lemma}\label{lem:Mpartition}
Let $n$ be an {\rm M}-number with $n_T$, $n_R$, $r$, $m_j$, $n'$ as above. Assume $\tau(n')\ge 4^j$. Then $[1,\tau(n)]$ can be partitioned into $\tau(n_T)\cdot\tau(m_j)$ sets $A_d$ (one for each $d\mid n_Tm_j$), each of size $\tau(n')$, with every element of $A_d$ coprime to $d$. Moreover each $A_d$ is a $K$-AP combination with common differences dividing $rm_j$, so for any $e\mid n'$ the count of elements of $A_d$ divisible by $e$ is within $K$ of $\tau(n')/e$.
\end{lemma}

\begin{proof}
We follow the same plan as Lemma \ref{lem:partition}, with two modifications to handle tight primes and primes of $m_j$ with exponent greater than $1$.

The tight primes are handled first and all at once using an argument akin to the one in the introduction with $M_p$. Since $v_p(n)=p-1$ for each prime $p\mid r$, the Chinese remainder theorem gives a bijection $\rho\colon D(n_T)\to\{0,1,\ldots,r-1\}$ defined by $\rho(d)\equiv v_p(d)\pmod p$
 for each $p\mid r$. We partition $[1,\tau(n)]=[1,r\tau(n_R)]$ into $\tau(n_T)=r$ residue classes modulo $r$, pairing $d$ with the class $\rho(d)\pmod r$. Each class is an arithmetic progression of length $\tau(n_R)$ and common difference $r$, and for any $a \equiv \rho(d)\pmod r$
 for $d\mid n_T$, the coprimality $(a,d)=1$ holds because $p\mid a$ if and only if $p\nmid d$ for each $p\mid r$. This plays exactly the role of the even/odd split on the prime $2$ in Lemma \ref{lem:partition}, producing $\tau(n_T)$ equal-length arithmetic progressions, one per divisor of $n_T$, with no contribution to the size of $K$.

The primes of $m_j$ are then handled inductively. After processing $p_1,\ldots,p_{i-1}$, we have $r\cdot\tau(m_{i-1})$ sets (one per divisor of $n_Tm_{i-1}$), each a $2^{i-1}$-AP combination of size $L'=\tau(m_j)\tau(n')/\tau(m_{i-1})$ with constituent APs having common differences coprime to $p_i$. We then split each current set $S$ into $a_i+1$ equal subsets of size $L'/(a_i+1)$, one for each $v_{p_i}(d)\in\{0,\ldots,a_i\}$, with those for $v_{p_i}(d)\ge1$ consisting of non-multiples of $p_i$. Choose a cutpoint $z$ so that the non-multiples of $p_i$ in $S$ before $z$ number exactly $a_i\cdot L'/(a_i+1)$. By Lemma \ref{lem:APcomb-error}, the count of $p_i$-multiples in $S$ is within $2^{i-1}$ of $L'/p_i$, hence at most $L'/p_i+2^{i-1}$. The M-number condition $a_i<p_i-1$ and the hypothesis $\tau(n')\ge4^j$ give
\[\frac{L'}{a_i+1}-\frac{L'}{p_i} \ge \frac{L'}{p_i(a_i+1)} \ge \frac{\tau(n')}{p_i} \ge \frac{4^j}{p_j} > 2^{j-1} \ge 2^{i-1},\]
using that $p_i-a_i-1\ge1$, $L'/(a_i+1)=\tau(m_j)\tau(n')/\tau(m_i)\ge\tau(n')$, and $p_i\le p_j$. Thus $L'/p_i+2^{i-1}<L'/(a_i+1)$, so there are more than $a_i L'/(a_i+1)$ non-multiples of $p_i$ in $S$ and the cutpoint $z$ is well-defined. The non-multiples before $z$ and the complementary set (multiples before $z$ together with everything after $z$) are each $2^i$-AP combinations by the same set-difference and union arguments as in Lemma \ref{lem:partition}. The non-multiples before $z$ are then divided into $a_i$ equal contiguous parts by $a_i-1$ further interval cuts; restricting a $2^i$-AP combination to $\{a\le z'\}$ or $\{a>z'\}$ preserves the AP-combination number (as in Lemma \ref{lem:partition}), so each part remains a $2^i$-AP combination. After all $j$ steps the sets are $2^j$-AP combinations of size $\tau(n')$, giving $K=2^j$.
\end{proof}

\begin{theorem}\label{thm:M}
Let $n$ be an {\rm M}-number with non-tight primes $p_1<p_2<\dots<p_\ell$ in increasing
order.  If $\ell\ge44$ and $v_{p_i}(n)=1$ for $\ell\ge i>j:= \lfloor\sqrt{\omega(n_R)}\rfloor$, then $n$ is matchable.
\end{theorem}

\begin{proof}
    
The proof follows closely that of Theorem \ref{th:sqfr}, using Lemma \ref{lem:Mpartition} in place of Lemma \ref{lem:partition}. 
Lemma \ref{lem:Mpartition} partitions $[1,\tau(n)]$ into sets $A_d$, one for each $d\mid n_Tm_j$, each of size $\tau(n')$ and forming a $K$-AP combination with common differences dividing $rm_j$. (The hypothesis $\tau(n')\ge4^j$ holds because $\tau(n')=2^{\ell-j}\ge2^{2j}=4^j$, since $\ell-j\ge2j$ once $\ell \geq 9$.)

The hypothesis that $n$ is not divisible by the square of any non-tight prime $q>p_j$ ensures that $n'$ is squarefree. For each $d\mid n_Tm_j$ we apply Hall's theorem to match $D(n')$ to $A_d$. Setting $\tilde\ell=\ell+1$ and $\tilde\jmath=j+1$ gives $K=2^{\tilde\jmath-1}$ and $\omega(n')=\tilde\ell-\tilde\jmath$, so the Hall's theorem arguments of Theorem \ref{th:sqfr} apply verbatim with $(\tilde\ell,\tilde\jmath)$ in place of $(\ell,j)$, covering all $\ell\ge44$. (The shift $\tilde\ell=\ell+1$ reflects that the prime $2$, here a tight prime contributing to $n_T$, plays the same structural role as $p_1=2$ in Theorem \ref{th:sqfr}. In both cases it contributes to neither $K$ nor to $f$, and the bound $f(n')\le\sum_{i=\tilde\jmath+1}^{\tilde\ell}1/P_i<0.93$ matches that theorem's bound precisely.)

Once we have obtained the matchings $\phi_d\colon D(n')\to A_d$, we can construct the coprime matching of $D(n)$ to $[1,\tau(n)]$ via $\psi(d\cdot e)=\phi_d(e)$, as in Theorem \ref{th:sqfr}.
\end{proof}

Since the set of M-numbers having fewer than any fixed number of non-tight prime factors has asymptotic density zero, as do those with a repeated large non-tight prime factor, we obtain the following corollary.
\begin{corollary}
\label{cor:M}
Every {\rm M}-number is matchable except possibly for a set of asymptotic density zero.
\end{corollary}
With Corollary \ref{cor:ud} we have the following.
\begin{corollary}
\label{cor:dens}
The set of matchable numbers has asymptotic density $\alpha$.
\end{corollary}

\section{Strongly matchable numbers}
Recall that we say $n$ is strongly matchable if for each coprime arithmetic progression of $\tau(n)$
integers there is a coprime matching to $D(n)$.
\begin{conjecture}
\label{conj}
A number $n$ is strongly matchable if and only if it is an {\rm M}-number.
\end{conjecture}
One would think that our techniques for matchable numbers could be applied here but there
is a difficulty.  In the proof of Theorem \ref{th:sqfr} we strongly used that we are mapping
$D(n)$ to an interval of small numbers, namely $[1,\tau(n)]$, but with strongly matchable,
the interval is not only generalized to a coprime arithmetic progression, it can be anywhere on 
the number line.  The latter condition is the difficulty.
We at least have a few results in the direction of the conjecture.
\begin{proposition}
\label{prop-strong}
Every strongly matchable number is an {\rm M}-number.
\end{proposition}
\begin{proof}
We prove the contrapositive. Suppose $n$ is not an M-number, and so $p^p\midd n$ for some prime $p$. Then at least $p/(p+1)$ of the members of $D(n)$ are divisible
by $p$.  Let $I$ be a coprime arithmetic progression of length $\tau(n)$.  Each shift $I+m$
is again a coprime arithmetic progression of length $\tau(n)$.  On average as $m$ varies, exactly
 $(p-1)/p$ of the integers in the set are coprime to $p$, so there is at least one such $m$
 where $I+m$ has less than $p/(p+1)$ of its members coprime to $p$.
We cannot coprimely match $D(n)$ with $I+m$, so $n$ is not strongly matchable.
\end{proof}
\begin{proposition}
\label{prop-strongfill}
If $n$ is strongly matchable and not divisible by the prime $p$, then
$p^{p-1}n$ is strongly matchable.
\end{proposition}
\begin{proof}
We have $\tau(p^{p-1}n)=p\tau(n)$. Let $I$ be a coprime arithmetic progression of length $p\tau(n)$, say $I=\{i_1,i_2,\dots,i_{p\tau(n)}\}$.
For $j=1,2,\dots,p$, let $I_j$ be the subsequence $(i_{j+kp})_k$ of length $\tau(n)$.
At most one of these subsequences has its terms all divisible by $p$, and the other subsequences are all coprime to $p$. If there is some $j$ where all the
terms of $I_{j}$ are divisible by $p$, denote this $j$ by $j_0$, otherwise let $j_0=1$.
There are coprime matchings $\psi_j$ from $D(n)$ to $I_j$ for each $j$. We construct a coprime matching from $D(p^{p-1}n)$ to $I$ as follows.
For $D(n)$, we already have $\psi_{j_0}$. For $p^kD(n)$, $1\le k\le p-1$, $k\ne j_0$,
we map it to one of the subsequences $I_j$ not used via $\psi_j$; that is, we map $p^kd\in p^kD(n)$ to $\psi_j(d)$. The union of these maps gives a coprime matching from $D(p^{p-1}n)$ to $I$, completing the proof.
\end{proof}
\begin{proposition}
\label{prop-strongdens}
The set of strongly matchable numbers has lower asymptotic density greater than $4/11$.
\end{proposition}
\begin{proof}
We first note that if $n$ is strongly matchable and $p$ is a prime with $p\nmidd n$ and $p>2\tau(n)$, then $pn$ is also strongly matchable. Indeed, take an arbitrary
coprime arithmetic progression $I$ of $\tau(pn)=2\tau(n)$ integers. There are coprime matchings of $D(n)$ to both the first half of $I$ and the second half of $I$, say $\psi_1,\psi_2$.
Not both of these
halves contain a multiple of $p$, so map $pD(n)$ to a half not containing a multiple of $p$
via $p\psi_i$, and use the unadorned injection for the other half.

Let $S_j$ be the set of odd squarefree numbers $s$ with $\omega(s)=j$ and each prime factor of $s$ is $< 2^{j}$, and let $S$ be the union of the sets $S_j$.
Suppose that $n>1$ is an odd squarefree number not divisible by any member of $S$. Then
\[
n=q_1q_2\dots q_k,\quad 3\le q_1<q_2<\dots <q_k,\quad \hbox{each }q_j\hbox{ prime},
\quad\hbox{each }q_j\ge2^{j}.
\]
Indeed, if not and $q_j< 2^{j}$ for some $j$, then
$q_1q_2\dots q_j\in S_j$, contradicting our assumption that $n$ is not divisible by any member of $S$.

We claim that any odd squarefree $n$ not divisible by any element of $S$ is strongly matchable.  First, $n=1$ is strongly matchable, and any odd prime $q$ is strongly matchable by the
argument above, since $q>2\tau(1)$.
 Now suppose that 
$n_j:=q_1q_2\cdots q_j$ is strongly matchable.   Since $\tau(n_j)=2^j$ and $q_{j+1}>2^{j+1}$, it
follows by the argument above that $n_jq_{j+1}$ is strongly matchable.  
Induction completes the argument that $n$ is strongly matchable.

It remains to show that such numbers $n$ comprise a set of positive
lower density. The reciprocal sum of the primes $q\le 2^{j}$ is $\log j+O(1)$,
and in fact from \cite[(3.18)]{RS} and a calculation, this sum is $<\log j$ when $j\ge4$.
Since $\sum_{p\le 7}1/p > 1.17619$, it follows that 
\[
\sum_{s\in S'_j}\frac1s\le\frac{(\log j-1.17619)^j}{j!},\quad j\ge4,
\]
where $S'_j$ is the set of $s\in S_j$ with all prime factors $>8$.
Let  $T$ denote the set of squarefree numbers
$n$ with all prime factors $>8$.   Then $T$ has an asymptotic density
equal to
\[
\frac
6{\pi^2}\prod_{p\le7}\left(1-\frac1p\right)\left(1-\frac1{p^2}\right)^{-1}=0.221640\dots\,.
\]
If $s\in S$ divides some $n\in T$, then each prime factor of $s$ is
$>8$, so for some $j\ge4$, we have $s\in S_j$, and so $s\in S_j'$.  The upper asymptotic
density of the set of multiples of the elements in $\cup_{j\ge4}S'_j$
is at most
\[
\sum_{j\ge4}\frac{(\log j-1.17619)^j}{j!}<0.000331239.
\]
Let $T'$ denote the subset of $T$ consisting of numbers not divisible by any member of $S$, so that
every member of $T'$ is strongly matchable.  The lower asymptotic density of $T'$ is greater than
$0.2213$.  We can boost this using Proposition \ref{prop-strongfill}. For each $n\in T'$ and each $p\in\{2,3,5,7\}$, we have $p\nmid n$ (since all prime factors of $n$ exceed $8$), so Proposition \ref{prop-strongfill} shows that $p^{p-1}n$ is also strongly matchable. Taking potential prime factors of this form into account improves the lower bound for the lower density of the set of strongly matchable numbers to at least
\[
0.2213\cdot\left(1+\frac{1}{2}\right)\left(1+\frac{1}{9}\right)\left(1+\frac{1}{5^4}\right)\left(1+\frac{1}{7^6}\right)
>0.3694>\frac{4}{11}.
\]
\end{proof}

\section*{Acknowledgments}
We thank Gerry Myerson for telling us about matchable numbers, and Bernardo Recam\'an
for his encouragement.  We are also grateful to the referee for many helpful comments.

\newpage

\scriptsize
\begin{longtable}{@{}rrrrrrrrrrr@{}}
\caption{Census of values of $\omega((s,n))$ for odd $s \in [1,2^{\ell+1}]$ where $n$ is the product of the first $\ell$ odd primes; large values are rounded up.}\label{tab:oddcensus}\\
\toprule
$\ell$ & $\omega_{\max}$\!\!\!\!\! & $c_{0}$ & $c_{1}$ & $c_{2}$ & $c_{3}$ & $c_{4}$ & $c_{5}$ &
$c_{6}$ & $c_{\geq 7}$ \\
\midrule
\endfirsthead

\toprule
$\ell$ & $\omega_{\max}$\!\!\! & $c_{0}$ & $c_{1}$ & $c_{2}$ & $c_{3}$ & $c_{4}$ & $c_{5}$ &
$c_{6}$ & $c_{\geq 7}$ \\
\midrule
\endhead

\bottomrule
\endfoot
3 & 2 & \num{3} & \num{4} & \num{1} & & & & & \\
4 & 2 & \num{7} & \num{7} & \num{2} & & & & & \\
5 & 2 & \num{13} & \num{11} & \num{8} & & & & & \\
6 & 3 & \num{25} & \num{21} & \num{17} & \num{1} & & & & \\
7 & 3 & \num{47} & \num{43} & \num{33} & \num{5} & & & & \\
8 & 3 & \num{89} & \num{95} & \num{56} & \num{16} & & & & \\
9 & 3 & \num{164} & \num{210} & \num{95} & \num{43} & & & & \\
10 & 4 & \num{309} & \num{441} & \num{176} & \num{94} & \num{4} & & & \\
11 & 4 & \num{597} & \num{878} & \num{376} & \num{179} & \num{18} & & & \\
12 & 4 & \num{1166} & \num{1736} & \num{798} & \num{341} & \num{55} & & & \\
13 & 5 & \num{2293} & \num{3376} & \num{1758} & \num{612} & \num{152} & \num{1} & & \\
14 & 5 & \num{4505} & \num{6612} & \num{3758} & \num{1138} & \num{364} & \num{7} & & \\
15 & 5 & \num{8897} & \num{12940} & \num{7892} & \num{2233} & \num{768} & \num{38} & & \\
16 & 5 & \num{17558} & \num{25510} & \num{16243} & \num{4553} & \num{1540} & \num{132} & & \\
17 & 6 & \num{34585} & \num{50650} & \num{32767} & \num{9755} & \num{2921} & \num{393} &
\num{1} & \\
18 & 6 & \num{68151} & \num{100919} & \num{65561} & \num{21015} & \num{5468} &
\num{1021} & \num{9} & \\
19 & 6 & \num{134282} & \num{201536} & \num{130617} & \num{45041} & \num{10387} &
\num{2370} & \num{55} & \\

20 & 6 & \num{264692} & \num{402354} & \num{260661} & \num{95132} & \num{20367} &
\num{5157} & \num{213} & \\
21 & 6 & \num{522290} & \num{803185} & \num{521116} & \num{197833} & \num{41521} &
\num{10487} & \num{720} & \\
22 & 7 & \num{1031482} & \num{1601975} & \num{1045031} & \num{405967} & \num{87226} &
\num{20541} & \num{2076} & \num{6} \\
23 & 7 & \num{2039192} & \num{3193416} & \num{2100106} & \num{825167} & \num{185875} &
\num{39443} & \num{5363} & \num{46} \\
24 & 7 & \num{4035965} & \num{6363543} & \num{4225826} & \num{1666753} & \num{396604} &
\num{75606} & \num{12689} & \num{230} \\
25 & 7 & \num{7992094} & \num{12678222} & \num{8506329} & \num{3361015} & \num{840077} & \num{147784} & \num{28017} & \num{894} \\
26 & 8 & \num{15830224} & \num{25256133} & \num{17121427} & \num{6780185} &
\num{1762060} & \num{297171} & \num{58718} & \num{2946} \\
27 & 8 & \num{31367217} & \num{50318652} & \num{34440697} & \num{13693052} &
\num{3657862} & \num{613215} & \num{118465} & \num{8568} \\
28 & 8 & \num{62163303} & \num{1.00e8} & \num{69234321} & \num{27697616} &
\num{7530374} & \num{1290469} & \num{233192} & \num{22820} \\
29 & 8 & \num{1.23e8} & \num{2.00e8} & \num{1.39e8} & \num{56080479} & \num{15404730} &
\num{2739542} & \num{454761} & \num{56236} \\
30 & 8 & \num{2.45e8} & \num{3.98e8} & \num{2.79e8} & \num{1.13e8} & \num{31329808} &
\num{5803412} & \num{887715} & \num{129521} \\
31 & 9 & \num{4.85e8} & \num{7.94e8} & \num{5.60e8} & \num{2.30e8} & \num{63616039} &
\num{12232414} & \num{1760793} & \num{283999} \\
32 & 9 & \num{9.64e8} & \num{1.58e9} & \num{1.12e9} & \num{4.65e8} & \num{1.29e8} &
\num{25579607} & \num{3565554} & \num{598677} \\
33 & 9 & \num{1.91e9} & \num{3.16e9} & \num{2.25e9} & \num{9.40e8} & \num{2.62e8} &
\num{53108556} & \num{7367039} & \num{1228836} \\
34 & 9 & \num{3.80e9} & \num{6.30e9} & \num{4.52e9} & \num{1.90e9} & \num{5.32e8} &
\num{1.09e8} & \num{15416636} & \num{2473356} \\
35 & 9 & \num{7.55e9} & \num{1.26e10} & \num{9.06e9} & \num{3.83e9} & \num{1.08e9} &
\num{2.24e8} & \num{32498453} & \num{4948093} \\
36 & 10 & \num{1.50e10} & \num{2.51e10} & \num{1.82e10} & \num{7.73e9} & \num{2.20e9} &
\num{4.59e8} & \num{68545789} & \num{9907571} \\
37 & 10 & \num{2.98e10} & \num{5.00e10} & \num{3.64e10} & \num{1.56e10} & \num{4.47e9} &
\num{9.36e8} & \num{1.44e8} & \num{19982844} \\
38 & 10 & \num{5.93e10} & \num{9.99e10} & \num{7.30e10} & \num{3.14e10} & \num{9.07e9} &
\num{1.91e9} & \num{3.01e8} & \num{40742931} \\
39 & 10 & \num{1.18e11} & \num{1.99e11} & \num{1.46e11} & \num{6.33e10} & \num{1.84e10} &
\num{3.89e9} & \num{6.26e8} & \num{83930709} \\
40 & 10 & \num{2.35e11} & \num{3.98e11} & \num{2.93e11} & \num{1.28e11} & \num{3.74e10} &
\num{7.93e9} & \num{1.29e9} & \num{1.74e8} \\
41 & 11 & \num{4.66e11} & \num{7.93e11} & \num{5.87e11} & \num{2.57e11} & \num{7.57e10} &
\num{1.62e10} & \num{2.66e9} & \num{3.64e8} \\
42 & 11 & \num{9.28e11} & \num{1.58e12} & \num{1.18e12} & \num{5.17e11} & \num{1.53e11} &
\num{3.29e10} & \num{5.46e9} & \num{7.61e8} \\

43 & 11 & \num{1.85e12} & \num{3.16e12} & \num{2.36e12} & \num{1.04e12} & \num{3.11e11} &
\num{6.71e10} & \num{1.12e10} & \num{1.59e9} \\
44 & 11 & \num{3.67e12} & \num{6.31e12} & \num{4.72e12} & \num{2.10e12} & \num{6.29e11} &
\num{1.37e11} & \num{2.29e10} & \num{3.32e9} \\
45 & 11 & \num{7.31e12} & \num{1.26e13} & \num{9.46e12} & \num{4.22e12} & \num{1.27e12} &
\num{2.79e11} & \num{4.68e10} & \num{6.91e9} \\
\end{longtable}

\newpage
\scriptsize
\begin{longtable}{@{}rrrrrrr@{}}
\caption{Selected gcd counts $\gcd_d = \#\{\text{odd } s \in [1,2^{\ell+1}] : (s,n)=d\}$ and auxiliary count $x_3 = \#\{\text{odd } s \in [1,2^{\ell+1}] : 3\mid s \text{ or } \omega((s,n))\geq 3\}$, where
$n$ is the product of the first $\ell$ odd primes; large values are rounded up.  Only values needed in the proof are shown; blank entries are not necessarily zero.}\label{tab:oddgcds}\\
\toprule
$\ell$ & $\gcd_{105}$ & $\gcd_{15}$ & $\gcd_{21}$ & $\gcd_{3}$ & $\mathrm{x_3}$ & $\gcd_{5}$ \\
\midrule
\endfirsthead

\toprule
$\ell$ & $\gcd_{105}$ & $\gcd_{15}$ & $\gcd_{21}$ & $\gcd_{3}$ & $\mathrm{x_3}$ & $\gcd_{5}$ \\
\midrule
\endhead

\bottomrule
\endfoot
3 & & & & \num{2} & & \\
4 & & & & \num{3} & & \\
5 & & \num{2} & & \num{5} & & \\
6 & & \num{3} & & \num{10} & & \\
7 & & \num{5} & & \num{21} & & \\
8 & & \num{9} & & \num{42} & & \\
9 & & \num{17} & & \num{86} & & \\
10 & & \num{36} & & \num{166} & & \\
11 & & \num{76} & & \num{315} & & \\
12 & & \num{152} & & \num{604} & & \\
13 & & \num{300} & & \num{1164} & & \\
14 & & \num{590} & & \num{2256} & & \\
15 & & \num{1139} & & \num{4416} & & \\
16 & & \num{2218} & & \num{8682} & & \\
17 & & \num{4314} & & \num{17139} & & \\
18 & & \num{8453} & & \num{33877} & & \\

19 & & \num{16639} & & \num{66979} & & \\
20 & & \num{32846} & & \num{132281} & & \\
21 & & \num{64979} & & \num{261372} & & \num{130677} \\
22 & & \num{128676} & & \num{516379} & & \num{258258} \\
23 & \num{42293} & \num{254834} & \num{169795} & \num{1020848} & & \num{510604} \\
24 & \num{83729} & \num{504881} & \num{336514} & \num{2019785} & \num{6334949} &
\num{1010179} \\
25 & \num{166004} & \num{1000144} & \num{666745} & \num{3998146} & \num{12706706} &
\num{1999526} \\
26 & \num{329275} & \num{1980869} & \num{1320714} & \num{7916785} & \num{25481743} &
\num{3958891} \\
27 & \num{653169} & \num{3923935} & \num{2616309} & \num{15683688} & \num{51096769} &
\num{7842216} \\
28 & \num{1295341} & \num{7773941} & \num{5183268} & \num{31078505} & \num{1.02e8} &
\num{15539148} \\
29 & \num{2568728} & \num{15406877} & \num{10272217} & \num{61607914} & \num{2.06e8} &
\num{30803093} \\
30 & \num{5097322} & \num{30566423} & \num{20378662} & \num{1.22e8} & \num{4.12e8} &
\num{61123344} \\
31 & \num{10115856} & \num{60661064} & \num{40441466} & \num{2.43e8} & \num{8.26e8} &
\num{1.21e8} \\
32 & \num{20080727} & \num{1.20e8} & \num{80288746} & \num{4.82e8} & \num{1.66e9} &
\num{2.41e8} \\
33 & \num{39866096} & \num{2.39e8} & \num{1.59e8} & \num{9.57e8} & \num{3.32e9} &
\num{4.78e8} \\
34 & \num{79187622} & \num{4.75e8} & \num{3.17e8} & \num{1.90e9} & \num{6.66e9} &
\num{9.50e8} \\
35 & \num{1.57e8} & \num{9.44e8} & \num{6.29e8} & \num{3.78e9} & \num{1.34e10} &
\num{1.89e9} \\
36 & \num{3.13e8} & \num{1.88e9} & \num{1.25e9} & \num{7.50e9} & \num{2.68e10} &
\num{3.75e9} \\
37 & \num{6.21e8} & \num{3.73e9} & \num{2.49e9} & \num{1.49e10} & \num{5.36e10} &
\num{7.46e9} \\
38 & \num{1.24e9} & \num{7.41e9} & \num{4.94e9} & \num{2.97e10} & \num{1.08e11} &
\num{1.48e10} \\
39 & \num{2.46e9} & \num{1.47e10} & \num{9.83e9} & \num{5.90e10} & \num{2.15e11} &
\num{2.95e10} \\
40 & \num{4.89e9} & \num{2.93e10} & \num{1.95e10} & \num{1.17e11} & \num{4.32e11} &
\num{5.86e10} \\
41 & \num{9.72e9} & \num{5.83e10} & \num{3.89e10} & \num{2.33e11} & \num{8.65e11} &
\num{1.17e11} \\
42 & \num{1.93e10} & \num{1.16e11} & \num{7.73e10} & \num{4.64e11} & \num{1.73e12} &
\num{2.32e11} \\
43 & \num{3.85e10} & \num{2.31e11} & \num{1.54e11} & \num{9.23e11} & \num{3.47e12} &
\num{4.62e11} \\
44 & \num{7.65e10} & \num{4.59e11} & \num{3.06e11} & \num{1.84e12} & \num{6.96e12} &

\num{9.18e11} \\
45 & \num{1.52e11} & \num{9.14e11} & \num{6.09e11} & \num{3.66e12} & \num{1.39e13} &
\num{1.83e12} \\
\end{longtable}

\newpage
\scriptsize
\begin{longtable}{@{}rrrrrrrrrrr@{}}
\caption{Census of values of $\omega((s,n))$ for $s \in [1,2^{\ell}]$ where $n$ is the product of the first $\ell$ odd primes; large values are rounded up.}\label{tab:censusinterval}\\
\toprule
$\ell$ & $\omega_{\max}$ & $c_{0}$ & $c_{1}$ & $c_{2}$ & $c_{3}$ & $c_{4}$ & $c_{5}$ &
$c_{6}$ & $c_{\geq 7}$ \\
\midrule
\endfirsthead

\toprule
$\ell$ & $\omega_{\max}$ & $c_{0}$ & $c_{1}$ & $c_{2}$ & $c_{3}$ & $c_{4}$ & $c_{5}$ &
$c_{6}$ & $c_{\geq 7}$ \\
\midrule
\endhead

\bottomrule
\endfoot

3 & 1 & \num{4} & \num{4} & & & & & & \\
4 & 2 & \num{6} & \num{9} & \num{1} & & & & & \\
5 & 2 & \num{11} & \num{18} & \num{3} & & & & & \\
6 & 2 & \num{22} & \num{30} & \num{12} & & & & & \\
7 & 3 & \num{44} & \num{51} & \num{32} & \num{1} & & & & \\
8 & 3 & \num{87} & \num{91} & \num{72} & \num{6} & & & & \\
9 & 3 & \num{171} & \num{180} & \num{138} & \num{23} & & & & \\
10 & 3 & \num{328} & \num{375} & \num{251} & \num{70} & & & & \\
11 & 4 & \num{626} & \num{793} & \num{451} & \num{174} & \num{4} & & & \\
12 & 4 & \num{1200} & \num{1646} & \num{847} & \num{381} & \num{22} & & & \\
13 & 4 & \num{2316} & \num{3359} & \num{1653} & \num{785} & \num{79} & & & \\
14 & 5 & \num{4510} & \num{6717} & \num{3407} & \num{1507} & \num{242} & \num{1} & & \\
15 & 5 & \num{8832} & \num{13321} & \num{7145} & \num{2823} & \num{639} & \num{8} & & \\
16 & 5 & \num{17400} & \num{26245} & \num{15033} & \num{5318} & \num{1494} & \num{46} & & \\
17 & 5 & \num{34338} & \num{51657} & \num{31407} & \num{10240} & \num{3247} & \num{183} & & \\
18 & 6 & \num{67840} & \num{101977} & \num{64647} & \num{20478} & \num{6606} & \num{595} & \num{1} & \\
19 & 6 & \num{134032} & \num{202022} & \num{131428} & \num{42198} & \num{12911} &

\num{1687} & \num{10} & \\
20 & 6 & \num{264639} & \num{401506} & \num{264780} & \num{88485} & \num{24820} &
\num{4281} & \num{65} & \\
21 & 6 & \num{522702} & \num{799799} & \num{530538} & \num{186141} & \num{47696} &
\num{9993} & \num{283} & \\
22 & 6 & \num{1032593} & \num{1595114} & \num{1060794} & \num{389735} & \num{93243} &
\num{21796} & \num{1029} & \\
23 & 7 & \num{2041220} & \num{3182621} & \num{2121272} & \num{808354} & \num{186821} &
\num{45105} & \num{3209} & \num{6} \\
24 & 7 & \num{4038813} & \num{6350266} & \num{4246629} & \num{1659760} & \num{382912} &
\num{89884} & \num{8900} & \num{52} \\
25 & 7 & \num{7995366} & \num{12665960} & \num{8515547} & \num{3381389} & \num{797705} & \num{175613} & \num{22568} & \num{284} \\
26 & 7 & \num{15832644} & \num{25252891} & \num{17098634} & \num{6854406} &
\num{1673928} & \num{341895} & \num{53269} & \num{1197} \\
27 & 8 & \num{31366915} & \num{50335662} & \num{34359252} & \num{13851260} &
\num{3510598} & \num{671330} & \num{118470} & \num{4241} \\
28 & 8 & \num{62157666} & \num{1.00e8} & \num{69065837} & \num{27961797} &
\num{7329469} & \num{1341027} & \num{251722} & \num{13225} \\
29 & 8 & \num{1.23e8} & \num{2.00e8} & \num{1.39e8} & \num{56440743} & \num{15199047} &
\num{2730221} & \num{515906} & \num{37449} \\
30 & 8 & \num{2.45e8} & \num{3.99e8} & \num{2.79e8} & \num{1.14e8} & \num{31245966} &
\num{5636269} & \num{1027549} & \num{97383} \\
31 & 8 & \num{4.85e8} & \num{7.95e8} & \num{5.60e8} & \num{2.30e8} & \num{63883643} &
\num{11755588} & \num{2023659} & \num{237081} \\
32 & 9 & \num{9.64e8} & \num{1.58e9} & \num{1.12e9} & \num{4.64e8} & \num{1.30e8} &
\num{24606119} & \num{3981833} & \num{546323} \\
33 & 9 & \num{1.91e9} & \num{3.16e9} & \num{2.25e9} & \num{9.38e8} & \num{2.64e8} &
\num{51472118} & \num{7908459} & \num{1205696} \\
34 & 9 & \num{3.80e9} & \num{6.30e9} & \num{4.52e9} & \num{1.90e9} & \num{5.36e8} &
\num{1.07e8} & \num{15904855} & \num{2563508} \\
35 & 9 & \num{7.55e9} & \num{1.26e10} & \num{9.06e9} & \num{3.83e9} & \num{1.09e9} &
\num{2.22e8} & \num{32484581} & \num{5314501} \\
36 & 9 & \num{1.50e10} & \num{2.51e10} & \num{1.82e10} & \num{7.73e9} & \num{2.21e9} &
\num{4.57e8} & \num{67141601} & \num{10821108} \\
37 & 10 & \num{2.98e10} & \num{5.00e10} & \num{3.64e10} & \num{1.56e10} & \num{4.47e9} &
\num{9.38e8} & \num{1.40e8} & \num{21818440} \\
38 & 10 & \num{5.93e10} & \num{9.98e10} & \num{7.30e10} & \num{3.14e10} & \num{9.07e9} &
\num{1.92e9} & \num{2.92e8} & \num{43904888} \\
39 & 10 & \num{1.18e11} & \num{1.99e11} & \num{1.46e11} & \num{6.33e10} & \num{1.84e10} &
\num{3.91e9} & \num{6.10e8} & \num{88614835} \\
40 & 10 & \num{2.35e11} & \num{3.98e11} & \num{2.93e11} & \num{1.28e11} & \num{3.73e10} &
\num{7.97e9} & \num{1.27e9} & \num{1.80e8} \\
41 & 10 & \num{4.66e11} & \num{7.93e11} & \num{5.87e11} & \num{2.57e11} & \num{7.57e10} &
\num{1.62e10} & \num{2.63e9} & \num{3.69e8} \\
42 & 11 & \num{9.28e11} & \num{1.58e12} & \num{1.18e12} & \num{5.18e11} & \num{1.53e11} &

\num{3.30e10} & \num{5.44e9} & \num{7.60e8} \\
43 & 11 & \num{1.85e12} & \num{3.16e12} & \num{2.36e12} & \num{1.04e12} & \num{3.10e11} &
\num{6.72e10} & \num{1.12e10} & \num{1.57e9} \\
44 & 11 & \num{3.67e12} & \num{6.31e12} & \num{4.72e12} & \num{2.10e12} & \num{6.28e11} &
\num{1.37e11} & \num{2.30e10} & \num{3.27e9} \\
45 & 11 & \num{7.31e12} & \num{1.26e13} & \num{9.46e12} & \num{4.22e12} & \num{1.27e12} &
\num{2.79e11} & \num{4.70e10} & \num{6.81e9} \\
\end{longtable}

\newpage
\begin{longtable}{@{}rrrrrrr@{}}
\caption{Selected gcd counts $\gcd_d = \#\{s \in [1,2^{\ell}] : (s,n)=d\}$ and auxiliary count $x_3 =
\#\{s \in [1,2^{\ell}] : 3\mid s \text{ or } \omega((s,n))\geq 3\}$, where $n$ is the product of the first $
\ell$ odd primes; large values are rounded up.   Only values needed in the proof are shown; blank entries are not necessarily zero.}\label{tab:gcdsinterval}\\
\toprule
$\ell$ & $\gcd_{105}$ & $\gcd_{15}$ & $\gcd_{21}$ & $\gcd_{3}$ & $\mathrm{x_3}$ & $\gcd_{5}$ \\
\midrule
\endfirsthead

\toprule
$\ell$ & $\gcd_{105}$ & $\gcd_{15}$ & $\gcd_{21}$ & $\gcd_{3}$ & $\mathrm{x_3}$ & $\gcd_{5}$ \\
\midrule
\endhead

\bottomrule
\endfoot
3 & & & & \num{2} & & \\
4 & & & & \num{4} & & \\
5 & & & & \num{7} & & \\
6 & & & & \num{11} & & \\
7 & & \num{7} & & \num{19} & & \\
8 & & \num{12} & & \num{37} & & \\
9 & & \num{20} & & \num{75} & & \\
10 & & \num{35} & & \num{154} & & \\
11 & & \num{68} & & \num{310} & & \\
12 & & \num{138} & & \num{612} & & \\
13 & & \num{280} & & \num{1195} & & \\
14 & & \num{566} & & \num{2320} & & \\
15 & & \num{1135} & & \num{4504} & & \\
16 & & \num{2241} & & \num{8783} & & \\
17 & & \num{4400} & & \num{17182} & & \\

18 & & \num{8607} & & \num{33788} & & \\
19 & & \num{16846} & & \num{66658} & & \\
20 & & \num{33048} & & \num{131710} & & \\
21 & & \num{65061} & & \num{260517} & & \num{130105} \\
22 & & \num{128425} & & \num{515466} & & \num{257532} \\
23 & & \num{254093} & \num{169296} & \num{1020164} & & \num{509903} \\
24 & \num{84153} & \num{503484} & \num{335445} & \num{2020025} & \num{6320878} &
\num{1009955} \\
25 & \num{166227} & \num{998109} & \num{665082} & \num{4000127} & \num{12695486} &
\num{2000257} \\
26 & \num{328933} & \num{1978571} & \num{1318640} & \num{7921325} & \num{25484436} &
\num{3961321} \\
27 & \num{651867} & \num{3922276} & \num{2614483} & \num{15691232} & \num{51130089} &
\num{7846881} \\
28 & \num{1292724} & \num{7774396} & \num{5182819} & \num{31088550} & \num{1.03e8} &
\num{15546187} \\
29 & \num{2564674} & \num{15411506} & \num{10274822} & \num{61618338} & \num{2.06e8} &
\num{30811418} \\
30 & \num{5092434} & \num{30577232} & \num{20386319} & \num{1.22e8} & \num{4.12e8} &
\num{61130339} \\
31 & \num{10111570} & \num{60679387} & \num{40455703} & \num{2.43e8} & \num{8.27e8} &
\num{1.21e8} \\
32 & \num{20079847} & \num{1.20e8} & \num{80309360} & \num{4.82e8} & \num{1.66e9} &
\num{2.41e8} \\
33 & \num{39872679} & \num{2.39e8} & \num{1.59e8} & \num{9.57e8} & \num{3.32e9} &
\num{4.78e8} \\
34 & \num{79206028} & \num{4.75e8} & \num{3.17e8} & \num{1.90e9} & \num{6.66e9} &
\num{9.50e8} \\
35 & \num{1.57e8} & \num{9.44e8} & \num{6.29e8} & \num{3.78e9} & \num{1.33e10} &
\num{1.89e9} \\
36 & \num{3.13e8} & \num{1.88e9} & \num{1.25e9} & \num{7.50e9} & \num{2.68e10} &
\num{3.75e9} \\
37 & \num{6.21e8} & \num{3.73e9} & \num{2.49e9} & \num{1.49e10} & \num{5.36e10} &
\num{7.46e9} \\
38 & \num{1.24e9} & \num{7.41e9} & \num{4.94e9} & \num{2.97e10} & \num{1.08e11} &
\num{1.48e10} \\
39 & \num{2.46e9} & \num{1.47e10} & \num{9.83e9} & \num{5.90e10} & \num{2.15e11} &
\num{2.95e10} \\
40 & \num{4.89e9} & \num{2.93e10} & \num{1.95e10} & \num{1.17e11} & \num{4.32e11} &
\num{5.86e10} \\
41 & \num{9.72e9} & \num{5.83e10} & \num{3.89e10} & \num{2.33e11} & \num{8.65e11} &
\num{1.17e11} \\
42 & \num{1.93e10} & \num{1.16e11} & \num{7.73e10} & \num{4.64e11} & \num{1.73e12} &
\num{2.32e11} \\
43 & \num{3.85e10} & \num{2.31e11} & \num{1.54e11} & \num{9.23e11} & \num{3.47e12} &
\num{4.62e11} \\

44 & \num{7.65e10} & \num{4.59e11} & \num{3.06e11} & \num{1.84e12} & \num{6.96e12} &
\num{9.18e11} \\
45 & \num{1.52e11} & \num{9.14e11} & \num{6.09e11} & \num{3.66e12} & \num{1.39e13} &
\num{1.83e12} \\
46 & \num{3.03e11} & \num{1.82e12} & \num{1.21e12} & \num{7.28e12} & \num{2.79e13} &
\num{3.64e12} \\
\end{longtable}

\normalsize


\bigskip

\begin{thebibliography}{99}
\bibitem{BP}
T. Bohman and F. Peng, Coprime mappings and lonely runners, Mathematika {\bf62} (2022), 784--804.

\bibitem{M}
N. McNew, Permutations and the divisor graph of $ [1, n]$, Mathematika {\bf63} (2023), 51--67.

\bibitem{P1}
C. Pomerance, Coprime matchings, Integers {\bf 22} (2022), \#A2, 9 pp.

\bibitem{P2}
C. Pomerance, Coprime permutations, Integers {\bf 22} (2022), \#83, 20 pp.

\bibitem{PS}
C. Pomerance and J. L. Selfridge, Proof of D. J. Newman's coprime mapping conjecture, Mathematika {\bf27} (1980), 69--83.

\bibitem{Re}
B. Recam\'an, Coprime matching of an integer's $d(n)$ divisors with the set of the first $d(n)$ integers,
mathoverflow, 2022.

\bibitem{R}
B. Rosser, The $n$th prime is greater than $n\log n$, Proc Lond. Math. Soc. (2),
{\bf45} (1939), 21--44.

\bibitem{RS}
J. B. Rosser and L. Schoenfeld, Approximate formulas for some functions of prime numbers,
Illinois J. Math. {\bf6} (1962), 64--94.

\bibitem{SS}
A. Sah and M. Sawhney, Enumerating coprime permutations, Mathematika {\bf68} (2022), 1120--1134.


\end{thebibliography}
\end{document}